\documentclass[reqno,11pt]{amsart}
\usepackage{amsmath,amssymb,amsfonts}

\usepackage{extpfeil}

\usepackage{tikz-cd}

\usepackage{tikz}
\usetikzlibrary{matrix,arrows}
\usepackage{verbatim}

\usepackage[pagebackref,colorlinks,citecolor={red!50!black},linkcolor={blue!80!black}]{hyperref}

\renewcommand*{\backref}[1]{}
\renewcommand*{\backrefalt}[4]{%
    \scriptsize%
    {
    \ifcase #1 (\textcolor{red}{Uncited.})%
          \or (Cited\ on p.~#2)%
          \else (Cited\ on pp.~#2)%
    \fi%
    }
}

\usepackage{orcidlink}

\usepackage{fourier}
\usepackage{inconsolata}
\usepackage[utf8]{inputenc}

\title{Colocalizing subcategories on differentially graded algebras}

\author[L. Alonso]{Leovigildo Alonso Tarr\'{\i}o \orcidlink{0000-0002-6896-0652}}
\address[L. A. T.]{CITMAGA\\
Departamento de Matem\'a\-ticas\\
Universidade de Santiago de Compostela\\
E-15782  Santiago de Compostela, Spain}
\email{leo.alonso@usc.es}

\author[A. Jerem\'{\i}as]{Ana Jerem\'{\i}as L\'opez \orcidlink{0000-0001-7964-1334}}
\address[A. J. L.]{CITMAGA\\
Departamento de Matem\'a\-ticas\\
Universidade de Santiago de Compostela\\
E-15782  Santiago de Compostela, Spain}
\email{ana.jeremias@usc.es}

\author[E. Loureiro]{Eduardo Loureiro Novo \orcidlink{0009-0009-6001-1928}}
\address[E. L. N.]{CITMAGA\\
Departamento de Matem\'a\-ticas\\
Universidade de Santiago de Compostela\\
E-15782  Santiago de Compostela, Spain}
\email{loureironovoe@gmail.com}

\thanks{This work has been partially supported by Xunta de Galicia's ED431C 2023/31 project with E.U.'s FEDER funds.}

\subjclass[2020]{18G80 (primary); 18E35, 18G35 (secondary)}

\date{May 25, 2025, 
 \emph{typeset}, \today}

\theoremstyle{plain}
\newtheorem{theorem}{Theorem}[section]
\newtheorem{proposition}[theorem]{Proposition}
\newtheorem{lema}[theorem]{Lemma}
\newtheorem{lemma}[theorem]{Lemma}
\newtheorem{corollary}[theorem]{Corollary}

\theoremstyle{remark}
\newtheorem{remark}[theorem]{Remark}
\newtheorem*{remark*}{Remark}

\theoremstyle{definition}
\newtheorem{definition}[theorem]{Definition}
\newtheorem*{ack}{Acknowledgements}
\newtheorem{cosa}[theorem]{}			

\numberwithin{equation}{theorem}

\newcommand{\CC}{{\mathcal C}}

\newcommand{\CF}{{\mathcal F}}
\newcommand{\CG}{{\mathcal G}}
\newcommand{\CH}{{\mathcal H}}

\newcommand{\CL}{\mathcal{L}}

\newcommand{\CP}{\mathcal{P}}

\newcommand{\CT}{\mathcal{T}}
\newcommand{\CV}{\mathcal{V}}

\newcommand{\CW}{\mathcal{W}}

\newcommand{\SC}{\mathsf{C}}
\newcommand{\SD}{\mathsf{D}}

\newcommand{\SL}{\mathsf{L}}
\newcommand{\Loc}{\mathsf{Loc}}

\renewcommand{\SS}{\mathsf{S}}

\newcommand{\ST}{\boldsymbol{\mathsf{T}}}

\newcommand{\D}{\boldsymbol{\mathsf{D}}}
\newcommand{\K}{\boldsymbol{\mathsf{K}}}
\newcommand{\LL}{\boldsymbol{\mathsf{L}}}
\newcommand{\R}{\boldsymbol{\mathsf{R}}}

\newcommand{\NN}{\mathbb{N}}
\newcommand{\ZZ}{\mathbb{Z}}

\newcommand{\ip}{{\mathfrak p}}
\newcommand{\iq}{{\mathfrak q}}
\newcommand{\ir}{{\mathfrak r}}
\newcommand{\im}{{\mathfrak m}}

 \newcommand{\hocolim}[1]{\begin{array}[t]{c} {\rm hocolim}\\[-7.5 pt]
 {\lto} \\[-7.5 pt] {\scriptstyle {#1}} \end{array}}
\newcommand{\dirlim}[1]{\begin{array}[t]{c} {\rm lim}\\[-7.5 pt]
 {\longrightarrow} \\[-7.5 pt] {\scriptstyle {#1}} \end{array}}

\newcommand{\lto}{\longrightarrow}

\newcommand{\liso}{\mathrel{\tilde{\lto}}}

\newcommand*{\longleftrightarrows}{
\ensuremath{%
\makebox[0pt][l]
{\raisebox{0.3ex}{\ensuremath{\longleftarrow}}}%
\raisebox{-0.3ex}{\ensuremath{\longrightarrow}}%
}}

\newcommand{\imp}{\Rightarrow}

\DeclareMathOperator{\Hom}{Hom}

\DeclareMathOperator{\rhom}{\R{}Hom}

\DeclareMathOperator{\cok}{Coker}

\DeclareMathOperator{\spec}{Spec}

\newcommand{\Ho}{\mathrm{H}}
\newcommand{\DGA}{\mathsf{DGA}}

\newcommand{\DGAM}{\mathsf{DGMod}}
\newcommand{\loc}{\mathsf{Loc}}
\newcommand{\col}{\mathsf{Coloc}}
\DeclareMathOperator{\ssupp}{supp}
\DeclareMathOperator{\cosupp}{cosupp}

\newcommand{\ie}{{\it i.e.~}}

\keywords{Derived categories, differential graded algebras, localization, colocalization.}

\begin{document}

\begin{abstract} Let $A$ be a bounded non positive commutative differential graded algebra $A$. Let $\D(A)$ its derived category of DG-modules. If $\D(A)$ is generated by the DG-modules corresponding to the residue fields of the ordinary ring $\Ho^0(A)$ then its localizing subcategories and its colocalizing subcategories are in bijection with the subsets of $\spec(\Ho^0(A))$. These results generalize well-known theorems by Neeman in \cite{Nct} and \cite{Ncs}, because any Noetherian ring satisfies this condition.
\end{abstract}

\maketitle
\tableofcontents

\section*{Introduction}

In his celebrated paper \cite{Nct} from 1992 Neeman proved that, for a Noetherian ring $A$, the localizing subcategories in the sense of Bousfield of its derived category $\D(A)$ are classified by the subsets of $\spec(A)$. The classification of localizing subcategories was generalized in 2004 to schemes in \cite{AJS3}, where the monoidal structure played a fundamental role. 

Every localizing subcategory of $\D(A)$ gives automatically a colocalizing subcategory by taking its right orthogonal. In 2011 Neeman showed \cite{Ncs} that these are the only colocalizing subcategories of $\D(A)$. A year later, the paper \cite{bik} opened a new point of view that allowed for a determination of colocalizing subcategories compatible with the monoidal structure by introducing the notion of cosupport, dualizing the notion of support in a derived category as defined by Neeman in \cite{Nct}. Armed with this point of view Verasdanis in his 2022 preprint \cite{v}, gave a classification of colocalizing categories for Noetherian schemes. Also  Barthel, Castellana, Heard, and Sanders in their 2023 preprint \cite{bchs} gave the corresponding result for the Balmer spectrum of general tt-categories under certain Noetherian hypothesis. 

In this line Shaul and Williamson \cite{SW} establish the classification of localizing and colocalizing subcategories of the derived category $\D(A)$ of a bounded non positive differential graded algebra under the hypothesis that the ring $\Ho^0(A)$ is Noetherian. Their method used a comparison between localizations in $\D(A)$ and in $\D(\Ho^0(A))$.

In this paper we give a generalization of Shaul and Williamson results, and, as a consequence, of the results by Neeeman, after all, an ordinary ring is an example of DG-algebra. Our approach is different from the previous authors inasmuch as there is a set of \emph{points} that determine the derived category. In precise terms a ring or a DG-algebra $A$ is called \emph{point generated} if the smallest localizing subcategory of its derived category $\D(A)$ containing all residue fields is the whole $\D(A)$. This property of $\D(A)$ is the key ingredient to obtain the classification results. Moreover, the proofs follow in a natural way form this condition. We show the classification of localizing subcategories in Theorem \ref{claslocalizantesdga} and the classification of colocalizaing subcategories in Theorem \ref{clascolocalizantesdga}. Moreover, we see that these classifications agree with the ones in $\D(\Ho^0(A))$. 

It is worth noting that a bounded non positive differential graded algebra $A$ such that $\Ho^0(A)$ is a Noetherian ring is point generated (Theorem \ref{k(p)generan}). In the tt-category setting, a main issue is to find the right notion of residue field, as is explained in \cite{bks}. There are plenty of examples of non Noetherian rings that are not point generated, see for one \cite{Oddball}. However recent results by Stevenson show that non Noetherian point generated rings abound (see \cite{St14}).

Let us now discuss briefly the contents of the paper. In the first section we gather the basic definitions of DG-algebras and their DG-modules. Its purpose is to fix the notation and give pointers to the literature. In section \ref{supresfi} we explain the DG-module structure of the residue fields of a non negative DG-algebra $A$ that turn out to correspond to the residue fields of $\Ho^0(A)$. With this notion a definition of support and cosupport is provided. In section \ref{pgDG} we define the notion of point generation and establish that the point generation of $\Ho^0(A)$ implies that of $A$.

The next two sections give the main results of this paper. Section \ref{clasloc} establishes in Theorem \ref{claslocalizantesdga} that for a a point generated commutative DG-algebra $A$ of finite amplitude, the localizing subcategories of its derived category $\D(A)$ are classified by the subsets of the spectrum of $\Ho^0(A)$, and also that they correspond to the localizing subcategories of $\D(\Ho^0(A))$. In the next section, \ref{clascoloc}, we prove the analogous result for colocalizing subcategories (Theorem \ref{clascolocalizantesdga}).

In the appendix we give a self contained proof of the fact that a Noetherian ring is point generated (Theorem \ref{LocK(x)generan}). This result is already present on \cite{coloc} in the setting of schemes.

\begin{ack}
We thank the comments of an anonymous referee.
\end{ack}

\section{Preliminaries}

We start recalling several well-known definitions to set the notations, for the convenience of the reader. We will mostly follow the treatment in \cite{Y1}.

\begin{cosa}\textbf{Differential graded rings}. 
A differential graded algebra is a graded ring $A=\bigoplus_{n\in\mathbb{Z}}A^n
$ together with a graded endomorphism $d\colon A\to A$ of degree 1, such that $d \circ d = 0$ and such that  the Leibniz rule holds, namely,
\[
d(ab)=d(a)\cdot b+(-1)^ia\cdot d(b),\quad a\in A^i,\ b\in A^j,\ i,j\in\mathbb{Z}.
\]
In other words $(A,d)$ is a complex such that the algebra structure
\[
A \otimes A \lto A
\]
is a homomorphism of complexes.

We say that $A$ is \emph{non positive} if $A^i=0$ for all $i>0$.

\emph{A homomorphism} of differential graded algebras is a graded ring homomorphism $f\colon A\to B$ that commutes with the differentials, \ie\/ $d_B \circ f = f \circ d_A$.

Differential graded algebras and their homomorphisms constitute a category that we will denote $\DGA$. The full subcategory of non positive differential graded algebras will be denoted by $\DGA^{\leq 0}$.

We will abbreviate the expression differential graded algebra by calling it simply DG-algebra.
\end{cosa}

\begin{cosa}
A DG-algebra is \emph{weakly commutative} if it satisfies $ba=(-1)^{ij}ab$ for all $a\in A^i,\ b\in A^j$. If furthermore $a^2=0$ for all $a\in A^i$ such that $i$ is odd, we say that $A$ is \emph{strongly commutative}. If $2$ is invertible in $A^0$, every weakly commutative DG-algebra is strongly commutative, so the difference is only relevant when the characteristic of $A^0$ is 2.
\end{cosa}

\begin{cosa}\textbf{Differential graded modules}.
A differential grade module (or simply DG-module) $M$ over a DG-algebra $A$ is a graded module (over the underlying graded algebra of $A$) together with a differential $d_M \colon M\to M$ of degree 1 satisfying
\[
    d_M(a\cdot m)=d_A(a)\cdot m +(-1)^i a\cdot d_M(m),\ \forall a\in A^i,\ m\in M.
\]
In other words, $M$ is a (cochain\footnote{We adopt throughout the convention of upper indices and differentials of positive degree 1. We will omit the adjective \emph{cochain} from now on.}) complex of abelian groups such that the action
\[
\cdot \colon A \otimes M \lto M
\]
is a homomorphism of complexes.

A homomorphism of graded modules is simply a homomorphism of complexes compatible with the respective actions of the DG-algebra $A$. We denote by $\DGAM(A)$ the category of DG-modules over $A$ and its homomorphisms.
\end{cosa}

\begin{cosa}\textbf{Homologies}.
Let $A$ be a DG-algebra and $M$ be a DG-$A$-Module. The coproduct of the homologies\footnote{In some contexts they are named \emph{cohomologies} due to the upper indices notation. We opt for the convention in \cite{yellow}, as no confusion may arise.} of $A$
 \[
    \Ho^\bullet(A)=\bigoplus_{i\in\mathbb{Z}}\Ho^i(A)
\]
is a graded ring and $\bigoplus_{i\in\mathbb{Z}}\Ho^i(M)$ is a graded $\Ho^\bullet(A)$-module.

We say that $A$ is \emph{bounded} if there exists $N \in \NN$ such that $\Ho^i(A) = 0$ whenever $|i| > N$, in other words, it has only a finite number of non zero homologies.
\end{cosa}

\begin{cosa}\textbf{The monoidal closed structure of DG-modules}.
The graded module of homomorphisms $\Hom_A^\bullet(M,N)$ that in degree $i \in \ZZ$ is formed by the graded homomorphisms of degree $i$ is endowed with the differential
 \[
d(\phi): =d_N\circ\phi - (-1)^j\phi\circ d_M,\ \forall \phi\in \Hom_A^j(M,N).
\]
It holds that
\[
   \Hom_{\DGAM(A)}(M,N):= \ker(d^0).
\]  
Analogously, the graded tensor product $M\otimes_A N$ is endowed with the differential
 \[
d(m\otimes n):=d_M(m)\otimes n+(-1)^i\cdot m\otimes d_N(n),\ \forall m\in M^i,\ n\in N^j
 \]
For more details, see \cite[Chapter 3]{Y1}. 
\end{cosa}

\begin{proposition}\label{tensorhomDG}
    Let $A$ and $B$ be weakly commutative DG-algebras, assume there is a homomorphism $A \to B$. Let $M \in\D(A)$ and $N, P \in\D(B)$. There is a natural isomorphism
    \[
   \Hom_A^\bullet(M, \Hom_B^\bullet(N, P)) \cong 
   \Hom_B^\bullet(M \otimes_{A} N, P)
    \]
\end{proposition}

\begin{proof} It follows from the usual graded tensor-hom adjunction. One checks that the involved differentials agree.
\end{proof}

\begin{cosa}\textbf{Truncations}.\cite[(1.4)]{Y2}\label{truncamientoint} 
For any $n\in\mathbb{Z}$ and $M$ en $\DGAM(A)$ define
\[
\tau^{\geq n}(M):=\ \cdots \to 0\to \cok(d^{n-1})\to M^{n+1}\to M^{n+2}\cdots 
\]
and
\[
\tau^{\leq n}(M):=\ \cdots\to M^{n-2}\to M^{n-1}\to \ker(d^n)\to 0\to\cdots 
\]

It holds that
\[
\Ho^i(\tau^{\geq n}M)=\begin{cases}
    \Ho^i(M)\ &\text{if }i\geq n \\
    0\ &\text{if } i<n
\end{cases}\qquad
\Ho^i(\tau^{\leq n}M)=\begin{cases}
    \Ho^i(M)\ &\text{if }i\leq n \\
    0\ &\text{if } i>n.
\end{cases}
\]
\end{cosa}

\begin{cosa}\textbf{Derived and homotopical categories of DG-modules}.\label{derivadaDGA}
For a DG-algebra $A$, the category $\DGAM(A)$ is analogous to the category of complexes of modules over a ring. Its derived category is constructed by inverting the \emph{quasi-iso\-mor\-phisms}, \ie morphisms that induce isomorphisms in homologies.

As in the classical case, it is most convenient to perform this localization in two steps. First we consider the homotopy category $\K(A)$ with the same objets as $\DGAM(A)$ and with morphisms defined by
\[
    \Hom_{\K(A)}(M,N)=\Ho^0(\Hom_{A}^\bullet(M,N)),
\]
\ie a morphism in $\K(A)$ is a class of homomorphism of DG-modules modulo those that are homotopic to zero. The categoy $\K(A)$ is no longer abelian but it has the strucutre of triangulated category, essentially because the cone construction carries over to this setting.

Finally, the derived category is defined as
\[
    \D(A):={\SS}^{-1}\K(A)
\]
where $\SS$ denotes the image of the class of quasi-isomorphisms in the homotopy category. It follows in an analogous fashion to the classic case that $\SS$ satisfies the axioms of a multiplicative system and that $\D(A)$ is also a triangulated category. The localization functor $Q\colon \K(A)\to\D(A)$ is a triangulated functor. Moreover, $\D(A)$ possesses arbitrary products and coproducts.

For full details, see \cite[Chapter 7]{Y1}.
\end{cosa}

\begin{theorem}\label{existenciaresolucionesDGA}
Let $A$ be a DG-algebra. The functor $Q\colon \K(A)\to\D(A)$ possesses left and right adjoints, in other words, every $M \in \K(A)$ possesses a K-projective and K-injective resolution.
\end{theorem}

\begin{proof}
 We point the readers to \cite[\S 11.4 and \S 11.6]{Y1}. 
\end{proof}

\begin{remark}
 A standard consequence of the previous Theorem is that every functor defined on $\K(A)$ admits left and right derived functors \cite[\S 8.3]{Y1}. A nice discussion of derived functors is is \cite[Chapter 2]{yellow}.
\end{remark}

In this vein we may extend Proposition \ref{tensorhomDG} to the derived setting.

\begin{theorem}\label{tensorhomDer}
    Let $A$ and $B$ be weakly commutative DG-algebras, assume there is a homomorphism $A \to B$. Let $M \in\D(A)$ and $N, P \in\D(B)$. There is a natural isomorphism
    \[
   \rhom_A^\bullet(M, \rhom_B^\bullet(N, P)) \cong 
   \rhom_B^\bullet(M \otimes_{A}^{\LL} N, P)
    \]
\end{theorem}

\begin{proof} 
    Take a K-projective resolution $Q \to N$ in $\DGAM(B)$ and a K-injective resolution $P \to I$. Notice that $\Hom_B^\bullet(Q, I))$ is K-injective and we have the following chain of isomorphisms
\begin{align*}
\rhom_A^\bullet(M, \rhom_B^\bullet(N, P)) \cong & \Hom_A^\bullet(M, \Hom_B^\bullet(Q, I)) \\
\cong \Hom^\bullet_B(M \otimes_A Q, I) \cong & \rhom_B^\bullet(M \otimes_{A}^{\LL} N, P)
\end{align*}
where the middle isomorphism follows from Proposition \ref{tensorhomDG} and the first and last by definition of derived functor.
\end{proof}

\begin{remark*}
 Full details are given in \cite[(12.10.12)]{Y1}.
\end{remark*}

Let $f\colon A \to B$ a homorphism of DG-algebras. Let us describe some basic functors between their corresponding derived categories. We will use a converse convention on super/subscripts with respect to variances to adhere the geometric convention.
The most basic functor is the forgetful functor

\[
f_* \colon \D(B) \lto \D(A)
\]

The functor $f_*$ is exact and possesses left and right adjoints defined by
\begin{align*}
    f^{\times}\colon \D(A)&\lto\D(B)& \LL f^*\colon\D(A)&\lto\D(B)\\  
    M&\longmapsto \rhom^\bullet_A(B,M) & M&\longmapsto M\otimes_A^{\LL}   B
\end{align*}

The assignment $(-)_*$ is obviously pseudo-functorial (see \cite[(3.6.5)]{yellow}), therefore so are $(-)^*$ and $(-)^\times$.

To be concrete, let us state the following.
\begin{proposition}\label{adjolvidosuperpor}
    Let $f\colon A\to B$ be a morphism in $\DGA$, for $M\in\D(A),\ N\in\D(B)$ there are natural isomorphisms:
    \begin{itemize}
        \item[(1)] $\Hom_{\D(A)}(M,f_*(N))\liso\Hom_{\D(B)}(\LL f^*M,N)$
        \item[(2)] $\Hom_{\D(A)}( f_*(N),M)\liso\Hom_{\D(B)}(N,f^\times M)$
    \end{itemize}
\end{proposition}
 \begin{proof}
    See \cite[12.6.6 (3)]{Y1} and \cite[12.6.12 (3)]{Y1}.
\end{proof}
 
\begin{cosa}\textbf{Projection  isomorphism.}\label{forproyeccióndga}
    For a homomorphism $f\colon A\to B$ in $\DGA^{\leq 0}$, and $M\in \D(A)$, $N\in\D(B)$, the canonical map
     \[
 f_*N\otimes^{\LL}_A M \liso  f_*(N\otimes^{\LL}_B \LL f^*M)     
     \]
is an isomorphism called the \emph{projection formula}.
\end{cosa}

The isomorphisms of Proposition \ref{adjolvidosuperpor} are upgraded to internal adjunctions. Specifically,
 
 \begin{itemize}
    \item[(1)]$\rhom^\bullet_A(M,f_*N) \liso f_*\rhom^\bullet_B(\LL f^*M,N)$
    \item[(2)]$\rhom^\bullet_A( f_*N,M) \liso f_*\rhom^\bullet_B(N,f^\times M)$.
 \end{itemize}
 
The first assertion is proved using Proposition \ref{adjolvidosuperpor}(1) and the monoidal character of $f^*$ and the second combining Proposition \ref{adjolvidosuperpor}(2) and  \ref{forproyeccióndga}. 

\begin{cosa}\textbf{Base change isomorphism}.\label{cambiobase_co}
Given a commutative square of DG-algebras
\begin{equation*}
\begin{tikzpicture}[baseline=(current  bounding  box.center)]
\matrix(m)[matrix of math nodes, row sep=2.6em, column sep=2.8em,
text height=1.5ex, text depth=0.25ex]{
  A & B \\
  C & B\otimes_AC\\
  };
\path[->,font=\scriptsize,>=angle 90]
(m-1-1) edge node[auto] {$f$} (m-1-2)
(m-2-1) edge node[auto] {$g$} (m-2-2)
(m-1-1) edge node[left] {$u$} (m-2-1)
(m-1-2) edge node[auto] {$v$} (m-2-2);
\end{tikzpicture}
\end{equation*}
the natural transformation
\begin{equation}\label{cambiobase*}
 \LL g^* \LL u^* f_* \cong \LL v^* \LL f^* f_*\lto \LL v^*,
\end{equation}
induced by the adjunction unit $\LL f^*\dashv f_*,$ yields a natural transformation 
\begin{equation}
\theta\colon\LL u^* f_*\lto g_*\LL v^*   \label{cambiobase*2},
\end{equation}
corresponding to the previous map through the adjunction $\LL g^*\dashv g_*$.

If $u\colon A\to C$ is flat, then the natural transformation (\ref{cambiobase*2}) is an isomorphism. Indeed, in this case, $v$ is also a flat morphism and the isomorphism (\ref{cambiobase*}),
\[
u^*f_*\cong g_*v^*,
\]
is given for each $M\in\D(B)$ by the following identifications
\[
g_*v^*M=g_*(M\otimes_B(B\otimes_A C))\cong f_*M\otimes_A C=u^*f_*M.
\]
\end{cosa}

\begin{cosa}
 We will mostly consider from now on \emph{non positive} DG-algebras, meaning that, as complexes of abelian groups their homologies are concentrated in degrees less or equal than $0$. These DG-algebras are called connective in the topological literature. We will denote by $\DGA^{\leq 0}$ the full subcategory of $\DGA$ of non-positive DG-algebras.

For $A \in \DGA^{\leq 0}$, there is a canonical epimorphism
\[
r\colon A \lto \Ho^0(A),
\]
together with the associated forgetful functor $r_* \colon \D(\Ho^0(A)) \to \D(A)$ and
 \begin{align*}
    r^{\times}\colon \D(\Ho^0(A))&\lto\D(A)& \LL r^*\colon\D(\Ho^0(A))&\lto\D(A)\\  
    M&\longmapsto \rhom^\bullet_A(\Ho^0(A),M) & M&\longmapsto M\otimes_A^{\LL}  \Ho^0(A)
\end{align*}
its adjoints.
\end{cosa}

\section{Supports and residue fields}\label{supresfi}

\begin{cosa}
Let $A$ be a commutative non positive DG-algebra $A$ and $\ip\in\spec(\Ho^0(A))$ a prime ideal. By $j_{\ip}\colon\Ho^0(A)\to k(\ip)$ we will denote the canonical homomorphism form $\Ho^0(A)$ to the residue field $k(\ip)$. We will denote by $i_{\ip}\colon A \to k(\ip)$ the composition $j_{\ip} \circ r$. This way there is a canonical diagram of homomorphisms associated to $\ip\in\spec(\Ho^0(A))$.
\end{cosa}

\begin{equation}\label{goodnotat}
\begin{tikzpicture}[baseline=(current  bounding  box.center)]
\matrix(m)[matrix of math nodes, row sep=2.6em, column sep=2.8em,
text height=1.5ex, text depth=0.25ex]{
  A & \Ho^0(A) &  k(\ip)\\
  };
\path[->,font=\scriptsize,>=angle 90]
(m-1-1) edge node[auto] {$r$} (m-1-2)
(m-1-2) edge node[auto] {$j_{\ip}$} (m-1-3)
(m-1-1) edge [bend left] node[above] {$i_{\ip}$} (m-1-3);
\end{tikzpicture}
\end{equation}

The next result gives two rules of orthogonality on residue fields in the context of differential graded algebras.  
\begin{lema}\label{tensorhomvanishesdga}
Let $\ip,\iq\in\spec(\Ho^0(A))$. The following are equivalent:
\begin{itemize}
    \item[(i)] $\ip\neq \iq,$
    \item[(ii)] $i_{\ip*}k(\ip)\otimes_A^{\LL}i_{\iq*}k(\iq)=0,$  
    \item[(iii)] $\rhom^\bullet_A(i_{\ip*}k(\ip),i_{\iq*}k(\iq))=0.$
\end{itemize}
\end{lema}

\begin{proof}
In (i) let us assume that there is an element $s\in\iq$ such that $s\notin\ip.$ Let $t\in A^0$ such that $r(t)=s.$ Consider the commutative diagram
\begin{equation*}
\begin{tikzpicture}[baseline=(current  bounding  box.center)]
\matrix(m)[matrix of math nodes, row sep=2.6em, column sep=2.8em,
text height=1.5ex, text depth=0.25ex]{
  A   & \Ho^0(A)   &  k(\ip)\\
  A_t & \Ho^0(A)_s &  k(\ip)\\
  };
\path[->,font=\scriptsize,>=angle 90]
(m-1-1) edge node[auto] {$r$} (m-1-2)
        edge node[auto] {$\tilde{u}$} (m-2-1)
(m-1-2) edge node[auto] {$j_{\ip}$} (m-1-3)
        edge node[auto] {$u$} (m-2-2)
(m-1-1) edge [bend left] node[above] {$i_{\ip}$} (m-1-3)
(m-2-1) edge node[auto] {$\tilde{r}$} (m-2-2)
(m-2-2) edge node[auto] {$j_{{\ip}_s}$} (m-2-3);
\draw[double distance = 1.5pt](m-1-3)--(m-2-3);
\end{tikzpicture}
\end{equation*}
Where $\tilde{u}$, $\tilde{r}$, $u$ and $j_{{\ip}_s}$ are the canonical localization homomorphisms. Both squares are cocartesian with flat vertical morphisms. So we have the following chain of isomorphisms
\begin{align*}
\tilde{u}_*\tilde{u}^* i_{\ip*}k(\ip) & \cong\tilde{u}_*\tilde{u}^* r_* j_{\ip *}k(\ip) \\
    & \cong\tilde{u}_*\tilde{r}_*u^*j_{\ip *}k(\ip) \\   
    & \cong\tilde{u}_*\tilde{r}_*j_{{\ip}_s*}k(\ip) \cong i_{\ip*}k(\ip)
\end{align*}
the first and the last isomorphisms hold by pseudofunctoriality of $(-)_*$,  and the rest by base change. We also have
\[
\tilde{u}^*i_{\iq *}k(\iq)=\tilde{u}^* r_*j_{\iq *}k(\iq)\cong\tilde{r}_* u^* j_{\iq *}k(\iq)=0,
\]
where the last equality $u^*j_{\iq *}k(\iq) = k(\iq)_s = 0$ holds because $s\in\iq$.

Finally, using that $\tilde{u}_*\tilde{u}^*i_{\ip*}k(\ip)\cong i_{\ip*}k(\ip)$, by the projection formula, we obtain 
\begin{align*}
     i_{\ip*}k(\ip)\otimes_A^{\LL} i_{\iq*}k(\iq)&\cong\tilde{u}_*\tilde{u}^* i_{\ip*}k(\ip)\otimes_A^{\LL} i_{\iq*}k(\iq)\\
     &\cong \tilde{u}_*(\tilde{u}^* i_{\ip*}k(\ip)\otimes_A^{\LL}\tilde{u}^*i_{\iq*}k(\iq))=0.
\end{align*}

Let us deal now with the implication (ii)$\imp$(iii). For $M\in\D(A)$  there is a canonical homomorphism
\[
\rhom^\bullet_A(M, i_{\iq*}k(\iq))\stackrel{\alpha}\lto\rhom^\bullet_A(M\otimes_A^{\LL} i_{\iq*}k(\iq), i_{\iq*}k(\iq)\otimes_A^{\LL}i_{\iq*}k(\iq))
\]
determined applying the functor $(-)\otimes_A^{\LL}i_{\iq*}k(\iq)$. Let $\beta$ be the homomorphism 
\[
\rhom^\bullet_A(M\otimes_A^{\LL} i_{\iq*}k(\iq), i_{\iq*}k(\iq)\otimes_A^{\LL}i_{\iq*}k(\iq))\stackrel{\beta}\lto\rhom^\bullet_A(M, i_{\iq*}k(\iq)),
\]
defined through the canonical homomorphisms $i_{\iq*}k(\iq)\otimes_A^{\LL}i_{\iq*}k(\iq)\to i_{\iq*}k(\iq)$ and $M \cong M\otimes_A^{\LL}A\to M\otimes_A^{\LL}i_{\iq*}k(\iq)$. The composition $\beta\circ\alpha$ is obviously the identity. Taking $M=i_{\ip*}k(\ip)$ and applying (ii) we conclude that  $\rhom_A(i_{\ip*}k(\ip),i_{\iq*}k(\iq))=0.$

Finally, (iii)$\Rightarrow$(i) follows from the fact that
\[
\rhom^\bullet_A(i_{\ip*}k(\ip),i_{\ip*}k(\ip))=i_{\ip*}\rhom^\bullet_{k(\ip)}(\LL i_{\ip}^*i_{\ip*}k(\ip),k(\ip))\neq0,
\]
because $\LL i_{\ip}^*i_{\ip *}k(\ip)\neq 0$.
\end{proof}

\begin{definition}\label{defco_sup}
    Let $M\in\D(A)$, the set
\[
\ssupp_A(M):=\{\ip\in\spec(\Ho^0(A))\ |\ i_{\ip*}k(\ip)\otimes_{A}^{\LL} M\neq 0\},
\]
is called \emph{support of $M$}. Similarly, the set
\[
\cosupp_A(M):=\{\ip\in\spec(\Ho^0(A))\ |\ \rhom^\bullet_A(i_{\ip*}k(\ip), M)\neq 0\}.
\]
is called \emph{cosupport of $M$}.
\end{definition}

Let us compute the support and cosupport in the simplest case.

\begin{lema}\label{soportekp}
    Let $\ip\in\spec(\Ho^0(A))$, it holds that
    \[
  \ssupp_A(i_{\ip*}k(\ip)) =   \cosupp_A(i_{\ip*}k(\ip)) = \{\ip\}
    \]
\end{lema}

\begin{proof}    
Follows immediately form Lemma \ref{tensorhomvanishesdga}.
\end{proof}

\section{Point generated DG-algebras}\label{pgDG}

Let us introduce notation for localizing and colocalizing subcategories.

\begin{definition}\label{defcoloc}
A triangulated subcategory $\SL$ of a triangulated category $\ST$ that is closed for coproducts is called a \emph{localizing subcategory}. Analogously a triangulated subcategory $\SC$ of $\ST$ is called a \emph{colocalizing subcategory} if it is closed for products.

A triangulated subcategory is called \emph{thick} if it contains the direct summands of its objects. It is a consequence of \emph{Eilenberg's swindle} that localizing and colocalizing subcategories are thick.
\end{definition}

\begin{cosa}
Denote by $\loc(S)$ the smallest localizing subcategory of a triangulated category $\ST$ that contains a set of objects $S\subset\ST$. We say that $\loc(S)$ is the localizing subcategory of $\ST$ \emph{generated} by $S$. If $S = \{M\}$ for $M \in \ST$ then we will abbreviate $\loc(M) = \loc(\{M\})$.

In a similar way we will denote by $\col(S)$ the smallest colocalizing subcategory of  $\ST$ 
that contains $S$.
\end{cosa}

\begin{cosa}\label{deforto}
Let $\SD$ be a subcategory of $\ST$, we define
\[
\SD^\perp := \{X\in\ST\ |\ \Hom(Y,X)=0,\ \forall\ Y\in \SD\}
\]
\[
^\perp\SD := \{X\in\ST\ |\ \Hom(X,Y)=0,\ \forall\ Y\in \SD\}
\]
We call $\SD^\perp$ the right orthogonal of $\SD$ and $^\perp\SD$ its left orthogonal.
\end{cosa}

We are interested in localizing subcategories of $\D(A)$ and of $\D(\Ho^0(A))$ where $A$ is a bounded commutative non positive DG-algebra $A$. We will use $\loc_{B}(S)$  and $\col_{B}(S)$, respectively, to distinguish between the localizing subcategories generated by a corresponding set $S$ in $\D(B)$ for any DG-algebra $B$.

\begin{lema}\label{restriccionlocalizantes}
    Let $f\colon A\to B$ be a homomorphism of DG-algebras and $M, N\in \D(B)$. If $M\in\loc_{B}(N)$ then $f_*M\in\loc_A(f_*N)$.
\end{lema}

\begin{proof}
It follows from the fact that the functor $f_*\colon\D(B)\to\D(A)$ is triangulated and preserves coproducts.
\end{proof}

\begin{remark}\label{observacion522}
 If $A$ is a non positive DG-algebra and  $r\colon A \to \Ho^0(A)$ is the canonical morphism, as a consequence of the previous result, for $M, N\in\D(\Ho^0(A))$, if $M$ belongs to $\loc_{\Ho^0(A)}(N)$ then $r_*M\in\loc_A(r_*N)$. 
\end{remark}

\begin{lema}\label{eldelostruncamientos}
If $A$ is a bounded commutative non positive DG-algebra of finite amplitude, then $A$ belongs to $\loc_A(r_*(\Ho^0(A)))$.
\end{lema}

\begin{proof}
Recall that for $n\in\mathbb{Z}$ its homology $\Ho^n(A)$ is an $\Ho^0(A)$-module, therefore $\Ho^n(A)$ belongs to $\loc_{\Ho^0(A)}(\Ho^0(A))$ so by the previous lemma we have that $\Ho^n(A)\in\loc_{A}(r_*(\Ho^0(A)))$. As the DG-algebra $A$ is non positive, it follows that $\tau^{\geq 0}A=\Ho^0(A)\in\loc_A(r_*(\Ho^0(A)))$.

For an integer $n<0$ assume that $\tau^{\geq n+1}A$ belongs to $\loc_A(r_*(\Ho^0(A)))$. Consider the following distinguished triangle
\[
  \tau^{\leq n}\tau^{\geq n}A\lto \tau^{\geq n}A
  \lto \tau^{\geq n+1}A \overset{+}\lto
\]
As $\tau^{\leq n}\tau^{\geq n}A\cong\Ho^{n}(A)$ and $\tau^{\geq n+1}A$ belong to $\loc_{A}(r_*\Ho^0(A))$ therefore $\tau^{\geq n}A$ belongs to $\loc_A(r_*(\Ho^0(A)))$. By induction it follows that $ \tau^{\geq n}A$ belongs to the subcategory $\loc_A(r_*(\Ho^0(A)))$ for all  $n\leq 0$, and thus $\tau^{\geq m}A=A\in\loc_A(r_*(\Ho^0(A))),$ where $m$ denotes the smallest integer such that $\Ho^m(A) \neq 0$.
\end{proof}

\begin{remark}
 The previous lemmas are also at the beginning of \cite[$\S$4]{SW}. We have included them couched in our notation and terminology for easy reference. 
\end{remark}

\begin{definition}\label{defpg}
Let $A$ be a non positive commutative DG-algebra. With the notation of \ref{goodnotat}, if 
\[
\D(A) = \loc_A\big(\big\{i_{\ip*}k(\ip)\ |\ \ip\in\spec(\Ho^0(A))\big\}\big)
\]
we say that $A$ is \emph{point generated}. Of course, the same definition applies to an ordinary commutative ring such as $\Ho^0(A)$.
\end{definition}

Our next task is to prove that $A$ is point generated whenever $\Ho^0(A)$ is. We begin by showing that for \emph{any} DG-algebra $A$, $A$ itself is a compact generator of its derived category, $\D(A)$.

\begin{lema}\label{Agencom}
The DG-algebra $A$ as a DG-module over itself is a compact generator of\/ $\D(A)$.
\end{lema}

\begin{proof}
    Let $M\in\D(A)$ be a DG-module over $A$. For all $n\in\mathbb{Z}$ it holds that 
    \[
    \Hom_{\D(A)}(A[n],M)=\Ho^n(M).
    \]
Therefore, if $\Hom_{\D(A)}(A[n],M)=0,$ for all $n\in\mathbb{Z}$, then $M=0,$ in other words, $A$ is a generator of  the category $\D(A)$. Moreover, the functor $\Hom_{\D(A)}(A,-)=\Ho^0(-)$ preserve coproducts, thus $A$ is compact.
\end{proof}

\begin{proposition}\label{compactamentegenerada}
  Let $\ST$ be a triangulated category with a compact generator $X$. Then, $\loc(X)=\ST$.
\end{proposition}

\begin{proof}
    If $X$ is a compact object of $\ST$, by \cite[1.7]{Ntty} the inclusion $\loc(X)\hookrightarrow\ST$ has a right adjoint. According to \cite[(Proposition 1.6)]{AJS1} (v) and (vi) this fact gives for any $M\in\ST$, a distinguished triangle
        \[
        Y_M\lto M\lto Z_M\overset{+}{\lto}
        \]
such that $Y_M \in \loc(X)$ and $Z_M \in  \loc(X)^\perp$. As $X$ is a generator of $\ST$, it follows that $\loc(X)^\perp=0$. Therefore $Z_M=0$ and $M\cong Y_M\in\loc(X)$.
\end{proof}

\begin{corollary}\label{compactamentegeneradacor}
    For any DG-algebra $A,$ it holds that
    \[
    \loc_A(A)=\D(A).
    \]
\end{corollary}

\begin{proof}
   It follows immediately from Lemma \ref{Agencom} and Proposition \ref{compactamentegenerada}.
\end{proof}

The following is the main result of this section.

\begin{theorem}\label{k(p)generan}
Let $A$ be a bounded commutative non positive DG-algebra such that $\Ho^0(A)$ is point generated. The smallest localizing subcategory of $\D(A)$ that contains   $i_{\ip*}k(\ip)$ for all  $\ip\in\spec(\Ho^0(A))$ is all of\/ $\D(A)$, in other words, $A$ is itself point generated.
\end{theorem}

\begin{proof}
    Let $\SL\subset\D(A)$ be the smallest localizing subcategory of $\D(A)$  generated by all the residue fields of $\Ho^0(A)$, \ie\/
\[
 \SL=\loc_A\big(\big\{i_{\ip*}k(\ip)\ |\ \ip\in\spec(\Ho^0(A))\big\}\big)
\]
    By the previous Proposition, to prove that $\SL = \D(A)$ it is enough to check that $A \in \SL$.

    By Lemma \ref{eldelostruncamientos}, we have that $A\in\loc_A(r_*(\Ho^0(A)))$. But $\Ho^0(A)$ is point generated therefore
\[
\D(\Ho^0(A)) =
\loc_{\Ho^0(A)}\big(\big\{j_{\ip*}k(\ip)\ |\ \ip\in\spec(\Ho^0(A))\big\}\big),
\]
so we are in position to apply Lemma \ref{restriccionlocalizantes} from which it follows that $r_*(\Ho^0(A))\in\SL$. Therefore, by Lemma \ref{eldelostruncamientos}, it holds that $A\in\loc_A(r_*\Ho^0(A))=\SL$.
\end{proof}

\section{Classifying localizing subcategories}\label{clasloc}

In this section we will classify the localizing subcategories of a bounded commutative non positive DG-algebra, assuming that $A$ is point generated.

\begin{lema}\label{tensoridealDGA}
    Let $A$ be a DG-algebra and let $\SL$ be a localizing subcategory of $\D(A)$. For any $M \in \SL$ and $N \in \D(A)$, it holds that  $M \otimes^{\LL}_A N \in \SL$.
\end{lema}

\begin{proof}
   Define the subcategory $\SD\subset\D(A)$ by
    \[
    \SD:=\big\{ N \in \D(A)\ |\ M \otimes_A^{\LL} N\in{\SL},\ \forall M\in\SL\big\}
    \]
The subcategory $\SD$ is a localizing subcategory of $\D(A)$ and contains $A$. By Corollary \ref{compactamentegeneradacor} it follows that $\SD=\D(A)$.
\end{proof}

\begin{remark}
 This can be restated saying that for any DG-algebras $A$, all localizing subcategories of $\D(A)$ are $\otimes$-ideals.
\end{remark}

\begin{lema}\label{tensorcero}
    Let $A$ be a bounded commutative non positive DG-algebra, $\SL\subset \D(A)$ a localizing subcategory and $\ip\in\spec(\Ho^0(A))$. In this case, $i_{\ip*}k(\ip)\in\SL$ if, and only if, there exists a  DG-module $M\in\SL$ such that $ \ip \in  \ssupp_A(M)$.
\end{lema}

\begin{proof}
    By Lemma \ref{tensorhomvanishesdga} the reverse implication follows immediately taking  $M = i_{\ip*}k(\ip)$. 
    For the direct implication, let $M \in \SL$ such that $i_{\ip*}k(\ip)\otimes_A^{\LL} M \neq 0$. By Lemma \ref{tensoridealDGA}, $i_{\ip*}k(\ip)\otimes_A^{\LL}M$ belongs to $\SL$, too. Moreover, by the projection formula  
   \[
   i_{\ip*}k(\ip)\otimes_A^{\LL} M 
   \cong i_{\ip*}(k(\ip)\otimes^{\LL}_{k(\ip)}\LL i_{\ip}^*M)
   = i_{\ip*}(\LL i_{\ip}^*M)
   \]
Thus $\LL i_{\ip}^*M$ is a non vanishing complex of $k(\ip)$-vector spaces, therefore it is isomorphic to a coproduct of shifts of $k(\ip)$ where $k(\ip)$ is regarded as a complex concentrated in a single degree. Notice that $i_{\ip*}$ preserves coproducts and that $\SL$ is thick. From this it follows that $i_{\ip*}k(\ip)$ belongs to $\SL$.
\end{proof}

\begin{proposition}\label{5210}
Let $A$ be a point generated bounded commutative non positive DG-algebra, and\/ $\SL\subset \D(A)$ a localizing subcategory. Then it holds that
    \[
    \SL=\loc_A\big(\big\{i_{\ip*}k(\ip)\ |\ i_{\ip*}k(\ip)\in\SL\big\}\big)
    \]
\end{proposition}

\begin{proof}
    Let $\widetilde{\SL}:=\loc_A\big(\big\{i_{\ip*}k(\ip)\ |\ i_{\ip*}k(\ip)\in\SL\big\}\big)$, and let $\SD\subset\D(A)$ be the localizing subcategory
    \[
    \SD:=\big\{ M\in \D(A)\ |\ M\otimes_A^{\LL} N\in\widetilde{\SL},\ \forall N\in\SL\big\}.
    \]
Let us check that $\SD$ contains all DG-modules $i_{\ip*}k(\ip),$ with $\ip\in\spec(\Ho^0(A))$. Indeed, If $i_{\ip*}k(\ip)\in \SL,$ by Lemma \ref{tensoridealDGA} then $i_{\ip*}k(\ip)\otimes_A^{\LL} N\in\widetilde{\SL}$ for all $N\in\D(A)$, in particular  for any $N\in\SL$, thus $i_{\ip*}k(\ip)\in \SD$. If on the other hand $i_{\ip*}k(\ip)\notin \SL$, by Lemma \ref{tensorcero} it holds that $i_{\ip*}k(\ip)\otimes_A^{\LL} N=0$ for all $N\in\SL$, therefore it belongs to $\widetilde{\SL}.$ All in all, $i_{\ip*}k(\ip)\in \SD$ for any $\ip\in\spec(\Ho^0(A))$. As $A$ is point generated by assumption, then $\SD=\D(A)$. We see that $A\in \SD$ and hence $\widetilde{\SL}=\SL$.
\end{proof}

\begin{cosa}
 For any set $X$, we denote by $\mathcal{P}(X)$ the set of susbsets of $X$, \ie\/ the powerset of $X$.
\end{cosa}

\begin{theorem}\label{claslocalizantesdga}
Let $A$ be a point generated bounded commutative non positive DG-algebra. The maps
\[
  \big\{\text{ Localizing subcategories of\/ }\D(A)\, \big\}\ \overset{\CL_A}{\underset{\CV_A}{\longleftrightarrows}} \ \mathcal{P}(\spec(\Ho^0(A)))
\]
defined by
\begin{align*}
\SL\longmapsto& \CV_A(\SL) = \big\{\ip\in\spec(\Ho^0(A))\ |\ i_{\ip*}k(\ip)\in\SL \big\}\\
V\longmapsto  &\CL_A(V)      = \loc_A\big(\big\{i_{\ip*}k(\ip)\ |\ \ip\in V\big\}\big)
\end{align*}
are inverse to each other.   
\end{theorem}

\begin{proof}
Denote $\CV=\CV_A$ and  $\CL=\CL_A$. By Proposition \ref{5210}, $\CL\circ\CV(\SL)=\SL$ for any localizing subcategory $\SL$ of $\D(A)$.

Given $V\subset\spec(\Ho^0(A))$, it is clear that $V\subseteq\CV\circ\CL(V)$. Let us check the reverse inclusion. Take a prime ideal $\iq$ of $\Ho^0(A)$ such that $\iq\notin V$. By Lemma \ref{tensorhomvanishesdga} $i_{\ip*}k(\ip)\otimes^{\LL}_A i_{\iq*}k(\iq)=0$ for any $\ip$ prime in $\Ho^0(A)$ such that $\ip \neq \iq$. Using Lemma \ref{tensorcero} we see that $i_{\iq*}k(\iq)$ does not belong to $\loc_A(\{i_{\ip*}k(\ip)\ |\ \ip\in V\}),$ therefore $\iq\notin\CV\circ\CL(V)$. Thus, $\CV\circ\CL(V)\subseteq V$, as wanted.
\end{proof}

\begin{remark*}
By Theorem \ref{LocK(x)generan} the theorem above holds whenever $A$ is a bounded commutative non positive DG-algebra with $\Ho^0(A)$ being Noetherian, in view of Theorem \ref{k(p)generan}.

This gives also an alternative proof of \cite[Theorem 2.8.]{Nct} when $A$ is an ordinary Noetherian ring.
\end{remark*}

The previous result can be restated in terms of supports, as defined in \ref{defco_sup}.

\begin{corollary}\label{clasificacionlocalizantes2}
Let $I$ be a class of objects of $\D(A)$ and $\SL = \loc(\{M \,|\, M \in I\})$ then,
\begin{itemize}
    \item[(1)] $\CV_A(\SL) = \bigcup_{M \in I}  \ssupp_A(M)$.
    \item[(2)] $N \in \SL \iff \ssupp_A(N) \subset \CV_A(\SL)$.
\end{itemize}
\end{corollary}

\begin{proof}
We begin proving (1). Let 
\[
V_I = \{\ip \in  \spec(\Ho^0(A) \,|\, \ip \in \ssupp_A(M) \text{ with } M \in I\}
\]
For any $M \in I$, $M \in \CL_A(V_I)$, therefore $\CL_A(V_I) = \SL$. By Theorem \ref{claslocalizantesdga}, $V_I = \CV_A(\SL)$.

Let us address (2). It follows from (1) that $\CV_A(\SL_N) = \ssupp_A(N)$ where $\SL_N := \loc_A(N)$. The subcategory $\SL$ is localizing, therefore $N \in \SL$ if, and only if, $\SL_N \subset \SL$. By Theorem \ref{claslocalizantesdga}, this gives $\ssupp_A(N) = \CV_A(\SL_N) \subset \CV_A(\SL)$.
\end{proof}

\begin{lema}\label{soporteasterisco}
Let $A$ be a bounded commutative non positive DG-algebra, for any DG-module $M\in\D(A)$ it holds that
    \[
    \ssupp_A(M) = \ssupp_{\Ho^0(A)}(\LL r^*M).
    \]
\end{lema}

\begin{proof}
For $\ip\in\spec(\Ho^0(A))$, it is enough to apply the projection isomorphism 
   \[
   M\otimes_A^{\LL}i_{\ip*}k(\ip)= M\otimes_A^{\LL}r_*j_{\ip*}k(\ip)\cong r_*(\LL r^*M\otimes_{\Ho^0(A)}^{\LL} j_{\ip*}k(\ip)),
    \]
to realize the claimed identity, because $r_*$ is a conservative functor.
\end{proof}

\begin{theorem} \label{localizantesarribaabajo} 
Let $A$ be a bounded commutative non positive DG-algebra such that $\Ho^0(A)$ is point generated. The correspondences

\begin{equation*}
\begin{tikzpicture}[baseline=(current  bounding  box.center)]
\matrix(m)[matrix of math nodes, row sep=3.5em, column sep=6em,
text height=1.5ex, text depth=0.25ex]{ \small
\left\{\begin{gathered}    
        \text{Localizing}\\
        \text{subcategories of\/ } \D(A)
    \end{gathered}\right\} & \small
\left\{\begin{gathered}    
        \text{Localizing}\\
        \text{subcategories of\/ } \D(\Ho^0(A))
    \end{gathered}\right\}  \\
  };
\draw [transform canvas={yshift= 0.4ex},font=\scriptsize,->]
(m-1-1) -- node[above]{$\sigma$} (m-1-2);
\draw [transform canvas={yshift=-0.4ex},font=\scriptsize,->]
(m-1-2) -- node[below]{$\rho$} (m-1-1);
\end{tikzpicture}
\end{equation*}
defined for any localizing subcategory $\SL_A$ of\/ $\D(A)$ and $\SL_{\Ho^0(A)}$ of\/ $\D(\Ho^0(A))$ by
\begin{align*}
    \sigma(\SL_A)&=\loc_{\Ho^0(A)}(\{\LL r^*M\ |\ M\in\SL_A\})\\
    \rho(\SL_{\Ho^0(A)})&=\loc_A(\{r_*N\ |\ N\in\SL_{\Ho^0(A)}\})
\end{align*}
are mutually inverse.
\end{theorem}

\begin{proof}
Let us complete the diagram in the statement with the correspondences given by Theorem \ref{claslocalizantesdga} for both algebras $A$ and $\Ho^0(A)$

\begin{equation*}
\begin{tikzpicture}[baseline=(current  bounding  box.center)]
\matrix(m)[matrix of math nodes, row sep=3.5em, column sep=1em,
text height=1.5ex, text depth=0.25ex]{ \small  
\left\{\begin{gathered}
        \text{Localizing}\\
        \text{subcategories of\/ } \D(A)
    \end{gathered}\right\} & & \small
\left\{\begin{gathered}    
        \text{Localizing}\\
        \text{subcategories of\/ } \D(\Ho^0(A))
    \end{gathered}\right\}  \\
 & \CP(\spec(\Ho^0(A)) &\\
  };
\draw [transform canvas={yshift= 0.4ex},font=\scriptsize,->]
(m-1-1) -- node[above]{$\sigma$} (m-1-3);
\draw [transform canvas={yshift=-0.4ex},font=\scriptsize,->]
(m-1-3) -- node[below]{$\rho$} (m-1-1);
\draw [shorten <=1.1cm, transform canvas={yshift= 0.4ex},font=\scriptsize,->]
(m-1-1) -- node[above, near end]{$\CV_A$} (m-2-2);
\draw [shorten >=1.1cm, shorten <=0.2cm, transform canvas={yshift=-0.4ex},font=\scriptsize,->]
(m-2-2) -- node[below, pos=0.4]{$\CL_A$} (m-1-1);
\draw [shorten <=1.1cm, transform canvas={yshift= 0.4ex},font=\scriptsize,->]
(m-1-3) -- node[above, near end]{$\CV_{\Ho^0(A)}$\,\,\,\,} (m-2-2);
\draw [shorten >=1.1cm, shorten <=0.2cm, transform canvas={yshift=-0.4ex},font=\scriptsize,->]
(m-2-2) -- node[below, pos=0.4]{\quad $\CL_{\Ho^0(A)}$} (m-1-3);
\end{tikzpicture}
\end{equation*}

By Theorem \ref{claslocalizantesdga} and its Corollary \ref{clasificacionlocalizantes2}, $\rho = \CL_A \circ \CV_{\Ho^0(A)}$ thus $\rho$ is a bijection.

We are left to prove that $\sigma$ is the inverse of $\rho$ and for this end, it is enough to see that $\CV_A(\SL_A) = \CV_{\Ho^0(A)}(\sigma(\SL_A))$. By Lemma \ref{soporteasterisco}
\[
\CV_A(\SL_A) 
= \bigcup_{M\in\SL_A}\ssupp_A(M)
= \bigcup_{M\in\SL_A}\ssupp_{\Ho^0(A)}(\LL r^* M)
\]
Finally, by Corollary \ref{clasificacionlocalizantes2}
\[
\bigcup_{M\in\SL_A}\ssupp_{\Ho^0(A)}(\LL r^*M)
= \bigcup_{N\in\sigma(\SL_A)}\ssupp_{\Ho^0(A)}(N)
= \CV_{\Ho^0(A)}(\sigma(\SL_A))
\] 
and this concludes the proof
\end{proof}

As a consequence, we see that $r_*$ preserves supports.

\begin{corollary}\label{soportepequeno}
    For every $N\in\D(\Ho^0(A))$ it holds that
    \[
    \ssupp_{\Ho^0(A)}(N)=\ssupp_A(r_*(N)).
    \]
\end{corollary}

\section{Classifying colocalizing subcategories}\label{clascoloc}

In the vein of Definition \ref{defcoloc}, given a family of objects $S\subset\D(A)$ we will denote by $\col_A(S)$ the smallest colocalizing subcategory of $\D(A)$ that contains $S$. 

\begin{lema}
    Let $A$ be a DG-algebra  and let $\SC$ be a colocalizing subcategory of $\D(A).$ For any  $N\in\SC$ and $M\in\D(A)$, it holds that $\rhom^\bullet_A(M,N)\in \SC$.
\end{lema}

\begin{proof}
    Let $\SL\subset\D(A)$ be the sucategory
    \[
    \SL:=\{M\in\D(A)\ |\ \rhom^\bullet_A(M,N)\in\SC,\ \forall N\in\SC\}
    \]
   The functor $\rhom^\bullet_A(-,N)$ transforms coproducts into products and $\SC$ is colocalizing, it holds that $\SL$ is a localizing subcategory of $\D(A)$. Moreover, $A\in\SL$, thus by Proposition \ref{compactamentegenerada} it follows that $\SL=\D(A)$.
\end{proof}

\begin{remark}
 This can be restated saying that any localizing subcategory of $\D(A)$ is an $\CH-$coideal.
\end{remark}

\begin{lema}\label{homcero}
    Let $A$ be a commutative DG-algebra, $\SC\subset\D(A)$ a colocalizing subcategory and $\ip\in\spec(\Ho^0(A))$. It holds that, $i_{\ip*}k(\ip)\in\SC$ if, and only if, there is a DG-module $N\in\SC$ such that $\rhom^\bullet_A(i_{\ip*}k(\ip),N)\neq 0$.
    
\end{lema}
\begin{proof}
   If $i_{\ip*}k(\ip)\in\SC$ it is enough to take $N=i_{\ip*}k(\ip)$ (see Lemma \ref{tensorhomvanishesdga}).

    For the reverse implication, let $N\in\SC$ such that $\rhom_A(i_{\ip*}k(\ip),N)\neq 0$. By the previous lemma, it holds that $\rhom^\bullet_A(i_{\ip*}k(\ip),N)$ also belongs to $\SC.$ From the natural isomorphism
\[
\rhom^\bullet_A(i_{\ip*}k(\ip),N) \cong 
i_{\ip*}\rhom^\bullet_{k(\ip)}(k(\ip),i_{\ip}^\times N) \cong 
i_{\ip*}i_{\ip}^\times N
\]
We deduce that $i_{\ip}^\times N$ is an acyclic complex of $k(\ip)$-vector spaces. If $i\in\mathbb{N}$ satisfies that $\Ho^i(i_{\ip}^\times N)\neq 0$ then $k(\ip)[-i]$ is a direct summand of $i_{\ip}^\times N$. As $i_{\ip*}$ preserves coproducts and $\SC$ is thick, we conclude that $i_{\ip*}k(\ip)\in\SC.$
\end{proof}

The next proposition is the key point to classify all colocalizing subcategories.

\begin{proposition}
Let $A$ be a point generated bounded commutative non positive DG-algebra. Let $\SC$ be a colocalizing subcategory of\/ $\D(A)$. It holds that
\[
\SC=\col_A(\{ i_{\ip*}k(\ip)\ |\ i_{\ip*}k(\ip)\in\SC\}).
\]
\end{proposition}
    
\begin{proof}
Let us denote $\widetilde{\SC}=\col_A(\{i_{\ip*}k(\ip)\ |\ i_{\ip*}k(\ip)\in\SC\})$. The full subcategory $\SD$ of $\D(A)$ defined by
    \[
    \SD=\{M \in\D(A)\ |\ \rhom^\bullet_A(M,N) \in \widetilde{\SC},\ \forall N \in \SC\}
    \]
is localizing, because the triangulated functor $\rhom^\bullet_A(-,N)$ takes coproducts into products and $\widetilde{\SC}$ is colocalizing. Let us see that $i_{\ip*}k(\ip)$ belongs to $\SD$ for any $\ip\in\spec(\Ho^0(A))$. Let $N\in\SC$, using the natural isomorphism
\[
\rhom^\bullet_A(i_{\ip*}k(\ip),N) \cong 
i_{\ip*}\rhom^\bullet_{k(\ip)}(k(\ip),i_{\ip}^\times N) \cong
i_{\ip*}i_{\ip}^\times N
\]
we deduce that $\rhom^\bullet_A(i_{\ip*}k(\ip),N)\in\col_A(i_{\ip*}k(\ip))$ because $i_{\ip}^\times N$ is a complex of $k(\ip)$-vector spaces and the functor $i_{\ip*}$ commutes with products. Therefore, if $i_{\ip*}k(\ip)\in \SC$, then $\rhom^\bullet_A(i_{\ip*}k(\ip),N) \in \widetilde{\SC}$ and so $i_{\ip*}k(\ip)\in\SD$. On the other hand, if $i_{\ip*}k(\ip) \notin \SC$, by the previous lemma, for all $N\in\SC$, $\rhom^\bullet_A(i_{\ip*}k(\ip),N)=0$ so it belongs to $\col_A(\widetilde{\SC})$. We conclude that $i_{\ip*}k(\ip)\in \SD.$

   Finally, we apply Theorem \ref{k(p)generan} to conclude that $\SD=\D(A)$ and, as a consequence, $\SC=\widetilde{\SC}$.
\end{proof}

\begin{theorem}\label{clascolocalizantesdga}
Let $A$ be a point generated bounded commutative non positive DG-algebra. The maps
\[
  \{\text{ Colocalizing subcategory of\/ }\D(A)\}\ \overset{\CC_A}{\underset{\CW_A}{\longleftrightarrows}} \ \mathcal{P}(\spec(\Ho^0(A)))
\]
defined by
\begin{align*}
    \SC\longmapsto&\CW_A(\SC)=\big\{\ip\in\spec(\Ho^0(A))\ |\ i_{\ip*}k(\ip)\in\SC \big\} \\
    W\longmapsto&\CC_A(W)=\col_A\big(\big\{i_{\ip*}k(\ip)\ |\ \ip\in W\big\}\big)
\end{align*}
are inverse to each other.
\end{theorem}

\begin{proof}
Denote $\CW=\CW_A$ and $\CC=\CC_A$. Let $\CC$ be a colocalizing subcategory of $\D(A)$. By the previous Proposition it follows that $\CC\circ\CW(\SC)=\SC$.

Suppose $W\subset\spec(\Ho^0(A))$. To see that $\CW\circ\CC(W)=W$, let us make explicit the definition of the left hand side:
    \[
\CW\circ\CC(W) = 
\{\ip\in\spec(\Ho^0(A)\ |\ i_{\ip*}k(\ip)\in\col_A(\{i_{\iq*}k(\iq)\ |\ \iq\in W\})\}
    \]
The inclusion $W\subseteq \CW\circ\CC(W)$ is obvious. Let us check now that $\CW\circ\CC(W)\subseteq W.$ Take $\iq\notin W$. By Lemma \ref{tensorhomvanishesdga}, $\rhom^\bullet_A(i_{\iq*}k(\iq),i_{\ip*}k(\ip))=0$ for all $\ip\in W$. By lemma \ref{homcero} we see that $i_{\iq*}k(\iq)\notin\col_A(\{i_{\ip*}k(\ip)\ |\ \ip\in W\})$, therefore $\rhom^\bullet_A(i_{\iq*}k(\iq),N) = 0$ for all $N \in \col_A(\{i_{\ip*}k(\ip)\ |\ \ip\in W\})$, thus $\iq\notin\CW\circ\CC(W).$
\end{proof}

\begin{remark*}
By Theorem \ref{LocK(x)generan} we see that the theorem holds whenever $A$ is bounded commutative non positive DG-algebra such that $\Ho^0(A)$ is \emph{Noetherian}, in view of Theorem \ref{k(p)generan}.

In particular, this applies when $A$ is an ordinary Noetherian ring, thus we get an alternate and somewhat simpler proof of \cite[Corollary 2.8.]{Ncs}.
\end{remark*}

As in the case of localizing subcategories, the previous result can be restated in terms of cosupports, as defined in \ref{defco_sup}.

\begin{lema}\label{soportesuperpor}
If $A$ is a commutative non positive DG-algebra, then 
    \[
    \cosupp_A(M)=\cosupp_{\Ho^0(A)}(r^\times M),
    \] 
for any $M\in\D(A)$   
\end{lema}

\begin{proof}
   The internal adjunction $r_*\dashv r^\times$, yields, for all $\ip\in\spec(\Ho^0(A))$ that
\[
r_{*}\rhom^\bullet_{\Ho^0(A)}(j_{\ip*}k(\ip),r^\times M) \cong 
\rhom^\bullet_A(i_{\ip *}k(\ip),M).
\]
The functor $r_*$ is conservative so the result follows.
\end{proof}

The following results are analogous to Theorem \ref{localizantesarribaabajo} and Corollary \ref{soportepequeno} for colocalizing subcategories.

\begin{theorem} Let $A$ be a bounded commutative non positive DG-algebra such that $\Ho^0(A)$ is point generated. The correspondences

\begin{equation*}
\begin{tikzpicture}[baseline=(current  bounding  box.center)]
\matrix(m)[matrix of math nodes, row sep=3.5em, column sep=6em,
text height=1.5ex, text depth=0.25ex]{  \small
\left\{\begin{gathered} 
        \text{Colocalizing}\\
        \text{subcategories of\/ } \D(A)
    \end{gathered}\right\} & \small
\left\{\begin{gathered}    
        \text{Colocalizing}\\
        \text{subcategories of\/ } \D(\Ho^0(A))
    \end{gathered}\right\} \\
  };
\draw [transform canvas={yshift= 0.4ex},font=\scriptsize,->]
(m-1-1) -- node[above]{$\tau$} (m-1-2);
\draw [transform canvas={yshift=-0.4ex},font=\scriptsize,->]
(m-1-2) -- node[below]{$\mu$} (m-1-1);
\end{tikzpicture}
\end{equation*}

defined by
\begin{align*}
    \tau(\SC_A)=\col_{\Ho^0(A)}(\{r^\times M\ |\ M\in\SC_A\})\\
    \mu(\SC_{\Ho^0(A)})=\col_A(\{r_*N\ |\ N\in\SC_{\Ho^0(A)}\})
\end{align*}
are mutually inverse.
\end{theorem}

\begin{proof}
The proof is analogous to Theorem \ref{localizantesarribaabajo}, using Lemma \ref{soportesuperpor}.
\end{proof}

\begin{corollary}
For all $M\in\D(\Ho^0(A))$  it holds that
    \[
\cosupp_{\Ho^0(A)}(M)=\cosupp_A(r_*(M)).
    \]
\end{corollary}

\appendix

\section{Noetherian rings are point generated}

\begin{cosa}
By Definition \ref{defpg}, a ring $A$ is \emph{point generated} if the set $\{\,k(\ip)\,/\,\ip \in \spec(A)\}$ is a family of generators of $\D(A)$ as a localizing subcategory.

In \cite{coloc} we gave a proof that if $A$ is Noetherian then it is point generated. In fact, our proof was given for a Noetherian scheme. The result for a Noetherian ring already follows by Neeman's \cite{Nct}. Here, we provide a direct proof of this fact in the algebraic setting for the convenience of the reader.
\end{cosa}

\begin{remark}
 It is a consequence of \cite[Theorem 5.7]{AJS1} that any ring $A$ (not necessarily Noetherian) is point generated precisely when a complex $M \in \D(A)$ is zero if, and only if, $\rhom^\bullet_A(k(\ip),M) = 0$ , for all $\ip \in \spec(A)$ .
\end{remark}

 \begin{lemma}\label{critpunt}
Let $A$ be a Noetherian ring and $\im$ a maximal ideal. For any $M \in \D(A)$, the complex $\R \Gamma_{\im}M\in \Loc_A{(\{k(\im)\})}$.
 \end{lemma}
 
 \begin{proof}
 As $A$ is Noetherian, the functor \emph{local cohomology} with respect to the ideal $\im$ is computed as 
  \[
  \Gamma_{\im}(-) =\dirlim{n > 0}\Hom_{A}(A/\im^n, -).
  \]
  Then,  using B\"okstedt-Neeman notion of countable homotopy colimits \cite[Definition 2.1]{BN}, we can express
  \[
  \R\Gamma_{\im}M \cong\! \hocolim{n > 0} \rhom^\bullet_A(A/\im^n, M),
  \]
for each $M \in \D(A)$.

It holds that $\R {i_{\im}}_*i_{\im}^\times M = \rhom^\bullet_A(A/\im^n, M)$, where $i_{\im} \colon A \to A/\im$ denotes the canonical morphism. Then it follows that if the complex $\rhom^\bullet_A(A/\im^n, M)$ belongs to $\Loc_A{(\{k(\im)\})}$, and so does $\rhom^\bullet_A(\im^n / \im^{n+1}, M)$ because the natural map
\[
\R {i_{\im}}_*i_{\im}^\times \rhom^\bullet_A(\im^n / \im^{n+1}, M)  \cong 
\rhom^\bullet_A(\im^n / \im^{n+1}, M) 
\]
is an isomorphism
for any $n\geq 1$. Starting with $n=1$, the existence of the distinguished triangles
\[\small
 \rhom^\bullet_A(A/\im^n, M) \lto
 \rhom^\bullet_A(A/\im^{n+1}, M), \CG) \lto 
 \rhom^\bullet_A(\im^n/\im^{n+1}, M) \overset{+}{\lto},
\]%
yields by induction that 
$
 \rhom^\bullet_A(A/\im^n, M) \in \Loc_A{(\{k(\im)\})},
$
for all $n\geq 1$. So, $\R \Gamma_{\im}M$ belongs to  $\Loc_A{(\{k(\im)\})}$.
\end{proof} 

\begin{remark}
Compare the previous Lemma with \cite[\S2]{Nct}.
\end{remark}

\begin{theorem}
\label{LocK(x)generan}
A Noetherian ring $A$ is point generated, \emph{i.e.} 
\[
\D(A) = \Loc_A(\{\,k(\ip)\,/\,\ip \in \spec(A)\}).
\]
\end{theorem}

\begin{proof} 
Set $\SL_A : = \Loc_A(\{\,k(\ip)\,/\,\ip \in \spec(A)\})$. For $M \in \D(A)$, let $\CF$ be the family of stable for specialization subsets $Z \subset \spec(A)$ such that $\R{}\Gamma_Z M \in \SL_A$. If $\{Z_\alpha\}_{\alpha \in \Lambda}$ is any chain of elements of $\CF$, and $M \to J$ is a K-injective resolution of $M$, by \cite[Theorem 2.2 and Theorem 3.1]{AJS1},
\[
\R{}\Gamma_{\cup Z_\alpha} M = \Gamma_{\cup Z_\alpha} J =
\dirlim{\alpha \in \Lambda} \Gamma_{Z_\alpha} J ,
\]
belongs to $\SL_A$, because  $\R{}\Gamma_{Z_\alpha}M = \varGamma_{Z_\alpha} J \in \SL_A$, for all $\alpha \in \Lambda$. Therefore  $\cup Z_\alpha \in \CF$. So there exists a maximal element $Z\in \CF$. Let $Y : = \spec(A) \setminus Z$, and suppose that $Y \neq \emptyset$. As $A$ is Noetherian, the family of closed subsets
\[
\CT = \Big\{\, \,\overline{\{\ir\}}\,\, \mid \,\,\ir\in \spec(A) \text{ such that } 
\overline{\{\ir\}} \cap Y \neq \emptyset \,\, \Big\}
\]
has a minimal element $\overline{\{\iq\}}\in \CT$. If $\ip \in \overline{\{\iq\}} \cap
Y$, then $\overline{\{\ip\}} \in \CT$, but
$\overline{\{\iq\}}$ is minimal, so $\ip = \iq$ and therefore $Z \cup \overline{\{\iq\}} = Z
\cup \{\iq\}$. 
Denote by $Z_\iq$ the set $\spec(A) \setminus \spec(A_\iq)$. There exists a distinguished triangle
\[
\R{}\Gamma_{Z_\iq} M \lto M \lto M_\iq \overset{+}{\lto}
\]
Apply the functor $\R{}\Gamma_{Z \cup \{\ip\}}$ to obtain the triangle
\[ 
\R{}\Gamma_{Z}  M \lto \R{}\Gamma_{Z \cup \{\ip\}} M 
\lto \R{}\Gamma_{\overline{\{\iq\}}} M_\iq \overset{+}{\lto},
\]
where, by Proposition~\ref{critpunt}, the third point $\R{}\Gamma_{\overline{\{\iq\}}} M_\iq$ belongs to $\Loc_A{(\{k(\iq)\})}$ so $\R{}\Gamma_{\overline{\{\iq\}}} M_\iq \in \SL_A$. As a consequence, $Z \cup \{\iq\}\in \CF$, contradicting the maximality of $Z$.
Necessarily, $Z = X$, therefore  $M \in \SL_A$.
\end{proof}

\begin{remark}
This theorem follows also from the results in \cite{Nct}, but the proof presented here is somewhat shorter. 

In the non Noetherian case, Neeman in \cite{Oddball} built a non Noetherian ring $R$ of dimension zero that is \emph{not point generated}. Despite these counterexamples, 
Stevenson has shown in \cite{St14} that there are plenty of non Noetherian rings that are point generated.
\end{remark}


\end{document}